\documentclass[reqno]{amsart}

\input xy
\xyoption{all}
\usepackage{accents}
\usepackage{epsfig}
\usepackage{color}
\usepackage{amsthm}
\usepackage{amssymb}
\usepackage{amsmath}
\usepackage{amscd}
\usepackage{amsopn}
\usepackage{graphicx}
\usepackage{url}
\usepackage{enumitem, hyperref}\hypersetup{colorlinks}

\usepackage{color} 

\definecolor{darkred}{rgb}{1,0,0} 
\definecolor{darkgreen}{rgb}{0,0.8,0}
\definecolor{darkblue}{rgb}{0,0,1}

\hypersetup{colorlinks,
linkcolor=darkblue,
filecolor=darkgreen,
urlcolor=darkred,
citecolor=darkgreen}

\makeatletter
\def\reflb#1#2{\begingroup
    #2%
    \def\@currentlabel{#2}%
    \phantomsection\label{#1}\endgroup
}
\makeatother

\numberwithin{equation}{section}
\newtheorem {Theorem}{Theorem}
\numberwithin{Theorem}{section}

\newtheorem {Lemma}[Theorem]    {Lemma}

\newtheorem {Proposition}[Theorem]{Proposition}

\theoremstyle{definition}

\theoremstyle{remark}
\newtheorem{Remark}[Theorem]{Remark}

\def    \eps    {\epsilon}

\newcommand{\supp}{\operatorname{supp}}

\newcommand{\id}{{\mathit id}}

\newcommand{\const}{{\mathit const}}

\newcommand{\tF}{\tilde{F}}

\newcommand{\hH}{\hat{H}}
\newcommand{\hf}{\hat{f}}
\newcommand{\hj}{\hat{j}}

\newcommand{\tf}{\tilde{f}}
\newcommand{\tj}{\tilde{j}}

\newcommand{\Cc}{{\mathcal C}}

\def    \C      {{\mathbb C}}
\def    \R      {{\mathbb R}}

\def    \12    {{\frac{1}{2}}}

\def    \p      {\partial}
\def    \codim  {\operatorname{codim}}

\def    \rk     {\operatorname{rk}}

\def    \U     {\operatorname{U}}

\def \hn   {\scriptscriptstyle{H}}

\begin{document}


\setlength{\smallskipamount}{6pt}
\setlength{\medskipamount}{10pt}
\setlength{\bigskipamount}{16pt}





\title[Fragility and Persistence of Leafwise Intersections]{Fragility
  and Persistence of Leafwise Intersections}

\author[Viktor Ginzburg]{Viktor L. Ginzburg}
\author[Ba\c sak G\"urel]{Ba\c sak Z. G\"urel}

\address{BG: Department of Mathematics,
University of Central Florida, 
Orlando, FL 32816, USA}
\email{basak.gurel@ucf.edu}

\address{VG: Department of Mathematics, UC Santa Cruz, Santa Cruz, CA
  95064, USA} \email{ginzburg@ucsc.edu}


\date{\today} 

\thanks{The work is partially supported by NSF grants DMS-1414685 (BG)
  and DMS-1308501 (VG)}

\bigskip

\begin{abstract}
  In this paper we study the question of fragility and robustness of
  leafwise intersections of coisotropic submanifolds. Namely, we
  construct a closed hypersurface and a sequence of Hamiltonians
  $C^0$-converging to zero such that the hypersurface and its images
  have no leafwise intersections, showing that some form of the
  contact type condition on the hypersurface is necessary in several
  persistence results.  In connection with recent results in
  continuous symplectic topology, we also show that $C^0$-convergence
  of hypersurfaces, Hamiltonian diffeomorphic to each other, does not
  in general force $C^0$-convergence of the characteristic foliations.

\end{abstract}

\maketitle

\tableofcontents

\section{Introduction and main results}
\label{sec:main-results}

\subsection{Introduction}
\label{sec:intro}

In this paper we study the question of fragility and existence of
leafwise intersections of coisotropic submanifolds. Our main result is
that leafwise intersections are fragile already for hypersurfaces and
need not exist even for $C^0$-small Hamiltonians: we construct a
closed hypersurface $M$ and a sequence of Hamiltonians $F_k$,
$C^0$-converging to zero, such that $M$ and its images
$\varphi_{F_k}(M)$ have no leafwise intersections. This shows, in
particular, that some form of the contact type condition on the
hypersurface $M$ is essential in Hofer's theorem, \cite{Ho}, stated
below and in its generalizations. Also, in
connection with the recent results from \cite{HLS,Op}, we prove that
the $C^0$-convergence of hypersurfaces Hamiltonian diffeomorphic to
each other does not, in general, imply $C^0$-convergence of their
characteristic foliations.

Let us now discuss our results in more detail and in a broader
context. Consider a closed coisotropic submanifold $M$ (e.g., a
hypersurface or a Lagrangian submanifold) of a symplectic manifold
$W$. Let $\varphi=\varphi_F$ be a compactly supported Hamiltonian
diffeomorphism of $W$, i.e., the time-one map of the flow generated by
a time-dependent Hamiltonian $F\colon S^1\times W\to \R$.  A
\emph{leafwise intersection} of $M$ and $\varphi(M)$ is a point $z\in
M$, or a pair $(z,\varphi(z))$, such that $\varphi(z)\in M\cap
\varphi(M)$ and moreover $z$ and $\varphi(z)$ lie on the same leaf of
the characteristic foliation of $M$. Thus leafwise intersections are
associated with $M$ and $\varphi$, and in the pair $(z,\varphi(z))$ it
is $\varphi(z)$ that is actually in $M\cap\varphi(M)$.

To the best of the authors' knowledge, leafwise intersections were
first considered in \cite{Mo}, and according to a theorem of Moser,
\cite{Mo}, and Banyaga, \cite{Ba}, leafwise intersections necessarily
exist when $M$ is a hypersurface and $\varphi$ is $C^1$-close to the
identity. In fact, this is true for any closed coisotropic submanifold,
as is easy to see by applying Weinstein's theorem on clean
intersections, \cite{We}, to the graph of the characteristic foliation
of $M$ near the diagonal; see \cite[p.\ 33]{Mo}.  Moreover, one has
the Lusternik--Schnirelmann and Morse type inequalities for the number
of leafwise intersections.  

Recently, Moser's theorem was strengthened by Ziltener in \cite{Zi:Mo},
where it was shown that leafwise intersections must exist for any
closed coisotropic submanifold $M$ whenever $\varphi$ is the time-one
map of a Hamiltonian isotopy $\varphi^t$ which is $C^0$-close to $\id$
or, more generally, when $\varphi^t(M)$ stays $C^0$-close to
$M$. Furthermore, in this case one still has the Lusternik--Schnirelmann
(cup-length) and Morse type multiplicity results. This result is in some sense
sharp since the condition that $\varphi$ is close to $\id$ is clearly
necessary unless $M$ meets some additional requirements.

Chronologically, however, the next crucial step after Moser's theorem
was a theorem of Hofer from \cite{Ho} (see also \cite{EH}) asserting
the existence of leafwise intersections for hypersurfaces in $\R^{2n}$
of restricted contact type, provided that $\varphi$ has sufficiently
small Hofer's norm
$$
\|\varphi\|_{\hn}:=\inf_{\varphi_{F}=\varphi}\| F\|_{\hn}, \textrm{
  where }
\| F\|_{\hn}=\int_{S^1}\left(\max F_t-\min F_t\right)\, dt.
$$
(Here we are assuming that $\varphi$ and $F$ are compactly supported.)
Moreover, in this theorem, the upper bound on $\|\varphi\|_{\hn}$ is
given by a certain homological capacity of the domain bounded by $M$;
see, e.g., \cite[Thm.\ 2.9]{Gi:coiso} for a symplectic topological
treatment of the question. Note also that in $\R^{2n}$ leafwise
intersections obviously need not exist when $\varphi$ is too far from
$\id$; for in this case we can easily have $M\cap
\varphi(M)=\emptyset$.

Since then the problem of existence of leafwise intersections has been
extensively investigated, and Hofer's theorem has been extended to
coisotropic submanifolds and to other ambient symplectic manifolds;
see, e.g., \cite{AF:rab,AF,AMc,AMo,Dr,Gi:coiso,Gu,Ka,Zi} for an
admittedly incomplete but representative list of results on leafwise
intersections. A common feature of these results is that, in contrast
with Moser's theorem, to ensure the existence of leafwise
intersections one has to impose some additional requirements on the
hypersurface or the coisotropic submanifold. This is usually a variant
of the contact type condition, but in \cite{Zi} leafwise intersections
are studied under the assumption that the characteristic foliation is
a fibration.

The main result of this paper (Theorem \ref{thm:main}) shows that some
assumption on $M$ is indeed necessary in Hofer's theorem to guarantee
the existence of leafwise intersections already when $M$ is a closed
hypersurface in $\R^{2n}$. It also shows that in Ziltener's theorem
one cannot replace $C^0$-norm by Hofer's norm. To be more specific, we
construct a closed smooth hypersurface $M\subset\R^{2n\geq 4}$,
$C^0$-close to the standard round sphere $S^{2n-1}$, and a sequence of
autonomous Hamiltonians $F_k$, $C^0$-converging to $0$ and supported
in the same compact set, such that $M$ and $\varphi_{F_k}(M)$ have no
leafwise intersections for all $k$. The proof relies heavily on the
construction of counterexamples to the Hamiltonian Seifert conjecture;
see, e.g., \cite{Gi:survey} and references therein.  Note that in
Theorem \ref{thm:main}, the convergence to zero is much stronger than
the convergence in Hofer's norm and not obviously related to the
$C^0$-convergence of the maps $\varphi_{F_k}$. (Apparently, the
sequence $\varphi_{F_k}$ we constructed does not $C^0$-converge; by
\cite{Zi:Mo}, it cannot $C^0$-converge to $\id$.)

Regarding the requirements on $M$, there is still a considerable gap
between what is currently known for hypersurfaces in $\R^{2n}$ and our
example. For instance, it is still not known if leafwise intersections
must exist when $M$ is a stable (in the sense of \cite{HZ}) closed
hypersurface in $\R^{2n}$ and $\|\varphi\|_{\hn}$ is sufficiently
small. The notion of stability can be extended to coisotropic
submanifolds (see \cite[Section 5]{Bo} for the original definition and
\cite{Gi:coiso} for a detailed discussion), and the question also
makes sense for closed coisotropic submanifolds.  Drawing from the
results in \cite{Us}, it seems reasonable to conjecture that the right
condition on $M$ for the existence of leafwise intersections of $M$
and $\varphi(M)$ when $\|\varphi\|_{\hn}$ is small is that the
characteristic foliation of $M$ is totally geodesic with respect to
some metric. However, as of this writing, this conjecture has far from
been proved. Without stability (or the totally geodesic condition) it
is not even known whether a closed coisotropic submanifold $M\subset
\R^{2n}$ with $1<\codim M<n$ must intersect $\varphi(M)$ when
$\|\varphi\|_{\hn}$ is small; cf.\ \cite{Gi:coiso, Ke, Us}.
 
Our second result concerns a different aspect of coisotropic
rigidity. An important question in the area, stemming from the analogy
between coisotropic and Lagrangian submanifolds, is whether or not a
smooth $C^0$-limit $M$ of smooth coisotropic submanifolds $M_k$ must
be coisotropic; cf.\ \cite{LS} for the Lagrangian counterpart. No
counterexamples are known even in the most general setting, but it is
not unreasonable to impose additional requirements on $M_k$ of two
types: stability or contact type conditions and that the submanifolds
$M_k$ are Hamiltonian diffeomorphic to each other. Under the latter
condition, one can also ask, provided that $M$ is indeed coisotropic,
if the characteristic foliations of $M_k$ converge to the
characteristic foliation of $M$. This second question is already of
interest when $M_k$ and $M$ are hypersurfaces and hence $M$ is
automatically coisotropic. (Note also that without the assumption that
the hypersurfaces $M_k$ are symplectomorphic, the characteristic
foliations need not converge, as is easy to see.) The answer to both
questions is affirmative when Hamiltonian diffeomorphisms between
$M_1$ and $M_k$ also $C^0$-converge to a homeomorphism; see \cite{HLS}
and also \cite{BO,Op}.

We show in Theorem \ref{thm:converge} that without this convergence
assumption or without extra assumptions on $M_k$ the answer to the
second question is negative for hypersurfaces. Namely, our proof of
Theorem \ref{thm:main} yields a sequence of hypersurfaces $M_k\subset
\R^{2n\geq 4}$ Hamiltonian diffeomorphic to each other and
$C^0$-converging to the round sphere $S^{2n-1}\subset \R^{2n}$, but
such that the characteristic foliations of $M_k$ do not $C^0$-converge
to the characteristic foliation on $M=S^{2n-1}$. More precisely, there
exists a sequence of closed characteristics $L_k\subset M_k$
$C^\infty$-converging to a simple closed curve in $S^{2n-1}$ which is
nowhere tangent to the characteristic foliation. Furthermore, the
characteristic foliation of $M_k$ is not homeomorphic to the
characteristic foliation on $S^{2n-1}$. The hypersurfaces $M_k$ are
diffeomorphic to $S^{2n-1}$ but not stable; see Remark
\ref{rmk:stability}.

\subsection{Main results}
\label{sec:result}
Before stating the main theorems of the paper, let us briefly recall
relevant definitions, some of which we have already used in Section
\ref{sec:intro}. Let $(W^{2n},\sigma)$ be a symplectic manifold; this
is just the standard symplectic $\R^{2n}$ in most of the results
considered here. Given a Hamiltonian $F\colon S^1\times W\to \R$,
which we will always assume to be compactly supported, we denote the
(time-dependent) Hamiltonian flow of $F$ by $\varphi_F^t$ and the
time-one map of this flow by $\varphi_F$. (Our sign convention for the
Hamiltonian vector field $\xi_F$ of $F$ is
$i_{\xi_{F_t}}\sigma=-dF_t$.) Throughout the paper, for the sake of
simplicity, all maps and functions are assumed to be $C^\infty$-smooth
unless explicitly stated otherwise.

Furthermore, recall from Section \ref{sec:intro} that, given a
coisotropic submanifold $M$ of $W$ (e.g., a hypersurface) and a
Hamiltonian diffeomorphism $\varphi=\varphi_F$, a \emph{leafwise
  intersection} of $M$ and $\varphi(M)$ is a point $z\in M$ such that
$\varphi(z)\in M\cap \varphi(M)$ and $z$ and $\varphi(z)$ lie on the
same leaf of the characteristic foliation of $M$. Here $\varphi(z)$,
rather than $z$, is actually an intersection of $M$ and
$\varphi(M)$. This, however, should cause no problem since $\varphi$
gives rise to a one-to-one correspondence between the leafwise
intersections $z$ and the points $\varphi(z)$. Sometimes we will also
refer to the pair $(z,\varphi(z))$ as a leafwise
intersection. (Although the definitions of leafwise intersections vary
between different papers, this one, arguably the most naive, is
sufficient for our purposes.) Leafwise intersections of $M$ and
$\varphi(M)$ depend on $M$ and the map $\varphi$, but only on the pair
of coisotropic submanifolds $M$ and $\varphi(M)$. We refer the reader
to, e.g., \cite{Gi:coiso} for a general discussion of coisotropic
submanifolds in the context of symplectic topology.

The main result of the paper is the following.

\begin{Theorem}
\label{thm:main} 
There exists a closed, smooth hypersurface $M\subset\R^{2n}$, $2n\geq
4$, and a sequence of $C^\infty$-smooth autonomous Hamiltonians
$F_k\stackrel{C^0}{\to} 0$, supported in the same compact set, such
that $M$ and $\varphi_{F_k}(M)$ have no leafwise intersections.
\end{Theorem}

Here the hypersurface $M$ cannot have contact type by the results of
\cite{Ho} and, in fact, $M$ is not even stable in the sense of
\cite{HZ}; see Remark \ref{rmk:stability} for a proof of this fact.

\begin{Remark}
  It readily follows from the proof that $M$ can be chosen to be
  diffeomorphic and arbitrarily $C^0$-close to the round sphere
  $S^{2n-1}$, and the Hamiltonians $F_k$ can also be chosen to be
  supported in an arbitrarily small neighborhood of $S^{2n-1}$. To be
  more precise, for any $\delta>0$, we can ensure that $M$ is the
  image of an embedding which is $\delta$-close to the standard
  embedding $S^{2n-1}\hookrightarrow\R^{2n}$ and that for all $k$ the
  Hamiltonians $F_k$ are supported in the $\delta$-neighborhood of
  $S^{2n-1}$ and $\varphi_{F_k}$ is also $\delta$-close to $\id$. As
  has been pointed out above, $\varphi_{F_k}$ cannot $C^0$-converge to
  $\id$ for a fixed~$M$ due to the results from \cite{Zi:Mo}.
\end{Remark}

As a byproduct of the proof of Theorem \ref{thm:main}, we obtain the
following result.

\begin{Theorem}
\label{thm:converge}
There exists a sequence of closed, smoothly embedded hypersurfaces
$M_k\subset \R^{2n}$, $2n\geq 4$, Hamiltonian diffeomorphic to each
other and $C^0$-converging and diffeomorphic to the round sphere
$S^{2n-1}\subset \R^{2n}$, but such that there exists a sequence of
closed characteristics $L_k\subset M_k$ $C^\infty$-converging to a
simple closed curve in $S^{2n-1}$ which is nowhere tangent to the
characteristic foliation and intersects every characteristic at at
most one point. Furthermore, the characteristic foliation on $M_k$ is
not homeomorphic to the characteristic foliation (the Hopf fibration)
on the sphere $S^{2n-1}$.
\end{Theorem}

The first assertion of the theorem should be understood as that the
characteristic foliations on $M_k$ do not $C^0$-converge to the
characteristic foliation on $S^{2n-1}$. (Here we leave aside a
somewhat delicate matter of defining $C^0$-convergence of foliations
(cf.\ \cite{Ep}) further complicated by the fact that in this context
$M_k$ are different, although diffeomorphic, manifolds.) The
assumption that the hypersurfaces $M_k$ are Hamiltonian diffeomorphic
to each other is essential -- without it, it is obvious that the
characteristic foliations on $M_k$ need not to converge to the
characteristic foliation on $M$ in any sense.

\subsection{Outline of the proofs}
\label{sec:outline}
The proofs of Theorems \ref{thm:main} and \ref{thm:converge} rely
heavily on the methods developed to construct counterexamples to the
Hamiltonian Seifert conjecture; see, e.g., \cite{Gi:survey}.

Let $S^{2n-1}$ be the unit sphere in $\R^{2n}$ with standard Darboux
coordinates, say, $(p_1,q_1,\ldots,p_n,q_n)$. Set $\tF=\eps\chi\cdot
p_1$, where $\chi$ is a cut-off function equal to one near
$S^{2n-1}$. Near $S^{2n-1}$, the map $\varphi_{\tF}$ is the parallel
transport by the vector $w=(0,\eps, 0,\ldots, 0)$. For $\eps>0$ small,
the only leafwise intersections of $S^{2n-1}$ and
$\varphi_{\tF}(S^{2n-1})$ are two points $z^\pm$ on the unit circle
$S$ in the $(p_1,q_1)$-plane, located near the North and the South
Poles on $S^{2n-1}$. (The points $\varphi_{\tF}(z^\pm)$ are the
intersections of $S$ and the transported circle $S+(0,\eps)$ in
$\R^2$.) Let us now insert two symplectic plugs into $S^{2n-1}$ to
interrupt $S$ between $z^+$ and $\varphi_{\tF}(z^+)$ and between $z^-$
and $\varphi_{\tF}(z^-)$ as in, e.g., \cite{Gi95,Gi:survey}; see Fig.\
\ref{fig:shift}. Here, however, since our goal is just to break the
characteristic $S$, we can use the plugs from \cite{Ci} with circular
cores.  Hence the only dimensional constraint is that $2n\geq 4$. We
choose the plugs narrow and thin, with width much smaller than
$\eps/2$, located in a very small neighborhood of the intersection of
$S^{2n-1}$ and the plane $q_1=0$, and place them in such a way that
they are displaced by $\varphi_{\tF}$. As a result, we obtain a new
hypersurface $M$ which is $C^0$-close to $S^{2n-1}$, differs from
$S^{2n-1}$ only within the plugs, and such that the characteristic $S$
is broken into several characteristics: one containing $z^\pm$ and
some other containing $\varphi_{\tF}(z^\pm)$.  This procedure is
illustrated in Fig.\ \ref{fig:shift}. We claim that $M$ and
$\varphi_{\tF}(M)$ have no leafwise intersections.  Indeed, the points
$z^\pm$ are no longer leafwise intersections for $M$ and
$\varphi_{\tF}(M)$, and since the plugs are displaced and due to the
plug-symmetry conditions, no new leafwise intersections are created.

\begin{figure}[h!]
\begin{center} 
\def\svgwidth{0.7\columnwidth}
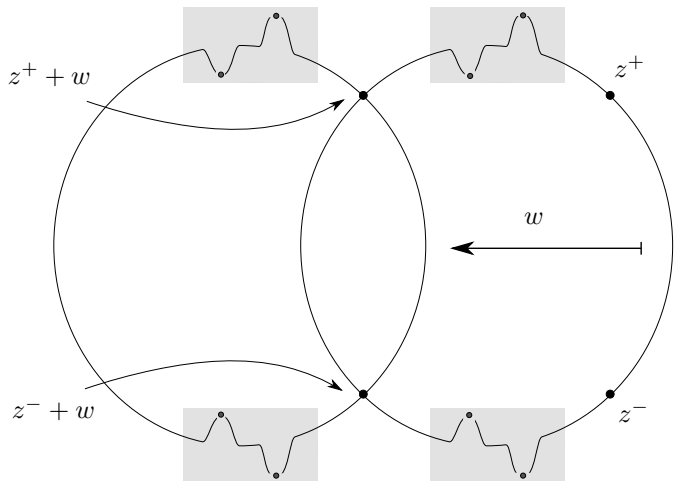
\caption{Breaking leafwise intersections.}
\label{fig:shift}
\end{center}
\end{figure}

Applying this construction to a sequence $\eps_k\to 0$, we obtain a
sequence of perturbations $M_k$ of $S^{2n-1}$ and a sequence of
Hamiltonians $\tF_k=\eps_k\chi\cdot p_1$ such that $M_k$ and
$\varphi_{\tF_k}(M_k)$ have no leafwise intersections and
$\tF_k\stackrel{C^\infty}{\to} 0$. Note also that the sequence $M_k$
can be chosen to $C^0$-converge to $S^{2n-1}$, and this is essential
for the proof of Theorem \ref{thm:converge}.

So far we have stayed close to the construction from \cite[Example
7.2]{Gi:coiso}. Now a crucial new step is the observation (Proposition
\ref{prop:plugs}) that the plugs can be inserted so that all
hypersurfaces $M_k$ are Hamiltonian diffeomorphic to $M=M_1$, i.e.,
there exists a sequence of Hamiltonian diffeomorphisms $\eta_k\colon
\R^{2n}\to\R^{2n}$ such that $\eta_k(M)=M_k$, and the maps $\eta_k$
are supported within the same compact set. (The key point is to find
``arbitrarily thin'' plugs Hamiltonian diffeomorphic to a given one. A
Hamiltonian diffeomorphism is constructed using Moser's method, which
ultimately reduces to a variant of a (singular) Cauchy problem for a
first-order PDE. We solve the Cauchy problem by the standard method of
characteristics, but extra care is needed at this step to account for
singularities.)  Then the Hamiltonians $F_k=\tF_k\circ\eta_k$ are also
supported within the same compact set. Clearly, $\varphi_{F_k}(M)$ and
$M$ have no leafwise intersections, and
$F_k\stackrel{C^0}{\to}0$. This proves Theorem \ref{thm:main} and also
Theorem \ref{thm:converge} with $M_k$ taken as the required sequence
of hypersurfaces.

This argument is essentially independent of the dimension $2n$, and
hence here we only detail it for $2n=4$. The general case can be
handled in a similar fashion. Namely, as in other Hamiltonian plug
constructions (see, e.g., \cite{Gi:survey}), one takes the product of
the lower--dimensional plug $P$ described here and the symplectic ball
$B^{2m}$ and equips $P\times B^{2m}$ with a $\U(m)$-invariant ``plug
two-form'' standard near the boundary of $P\times B^{2m}$.

The proof is organized as follows. In Section \ref{sec:plugs}, we
describe the plugs, state Proposition \ref{prop:plugs}, and derive
Theorems \ref{thm:main} and \ref{thm:converge} from the
proposition. Proposition \ref{prop:plugs} is then proved in Section
\ref{sec:pf_plugs}.

\section{Symplectic plugs in $\R^4$}
\label{sec:plugs}

\subsection{Plugs}
In this section we discuss the construction of a symplectic plug in
the setting specifically tailored to the proof of Theorem
\ref{thm:main}; we refer the reader to, e.g., \cite{Gi:survey} for a
treatment of the plugs in a much more general context. Let $\Pi=
[-\delta,\,\delta]\times [-T,\,T]$, for some $\delta >0$ and $T>0$,
with coordinates $(x,t)$ and let $P=S^1\times \Pi$. We denote the
angle coordinate on $S^1$ by $\theta$. For two auxiliary functions $f$
and $H$ on $\Pi$ to be specified later, set
\begin{equation}
\label{eq:form}
\omega=d(H\,d\theta-f\,dt) .
\end{equation}
Furthermore, consider the product $B= P\times [-a,\,a]$ for some
$a>0$, and denote by $y$ the coordinate on $[-a,\,a]$. Thus we have
$$
B=
\underbrace{{\strut S^1}}_{\theta}
\times
\underbrace{\strut [-\delta,\,\delta]}_{x} 
\times 
\underbrace{\strut [-T,\,T]}_{t}
\times 
\underbrace{\strut  [-a,\,a]}_{y}.
$$
We equip $B$ with the symplectic form
$$
\sigma=dx\wedge d\theta + dy\wedge dt
$$
and identify $P$ with the subset $y=0$ of $B$.

\begin{Lemma} 
\label{lemma:non-deg}
Assume that $f'_x$ and $H'_x$ do not vanish simultaneously, $|f|<a$
and $|H|\leq \delta$, and that $f\equiv 0$ and $H\equiv x$ near
$\p\Pi$. Then $\omega$ is a maximally non-degenerate form on $P$ with
characteristic vector field
\begin{equation}
\label{eq:X}
X=f'_x \frac{\p}{\p \theta} -H'_t\frac{\p}{\p x} +H'_x\frac{\p}{\p t}
=f'_x \frac{\p}{\p \theta}+\xi_H,
\end{equation}
where $\xi_H$ is the Hamiltonian vector field of $H$ on $(\Pi,dx\wedge
dt)$.  In other words, $\rk\omega=2$ and $i_X\omega=0$ and $X\neq 0$.
Furthermore, there exists an embedding $j\colon P\to P\times (-a,\,a)$
such that $j=(\id,0)$ near $\p P$ and $j^*\sigma=\omega$.
\end{Lemma}

\begin{proof} Set 
\begin{equation}
\label{eq:j}
j(\theta,x,t)=\big(\theta,H(x,t),t,-f(x,t)\big).
\end{equation}
Clearly, $j=(\id,0)$ near $\p P$, since $f\equiv 0$ and $H\equiv x$
near $\p\Pi$, and $j^*\sigma=\omega$.  It readily follows from the
assumption that $f'_x$ and $H'_x$ do not vanish simultaneously that
$j$ is an embedding. Hence $j^*\sigma$ is maximally
non-degenerate. Finally, a direct calculation shows that $i_X\omega=0$
and $X\neq 0$.
\end{proof}

We require the functions $H$ and $f$ to meet the following conditions:

\begin{itemize}

\item[\reflb{(P1)}{(P1)}] $H\equiv x$ and $f\equiv 0$ near $\p \Pi$,
  and $|f|<a$ and $|H|\leq \delta$ on $\Pi$;

\item[\reflb{(P2)}{(P2)}] $H'_x\geq 0$;

\item[\reflb{(P3)}{(P3)}] the critical points of $H$ are
  $p_\pm=(0,\tau_\pm)$ and $f'_x(p_\pm)\neq 0$, while $H'_x =0$ only
  at these points;

\item[\reflb{(P4)}{(P4)}] $H$ is even in $t$ and $f$ is odd in $t$ for
  any $x$; in particular, $\tau_-=-\tau_+$.

\end{itemize}
Note that these requirements include the conditions of Lemma
\ref{lemma:non-deg}. It is easy to see that such functions $H$ and $f$
do exist for any positive parameters $\delta$, $T$ and~$a$. The level
sets of $H$ and its Hamiltonian flow are shown in Fig.\
\ref{fig:plug}.

\begin{figure}[h!]
\begin{center} 
\def\svgwidth{0.6\columnwidth}
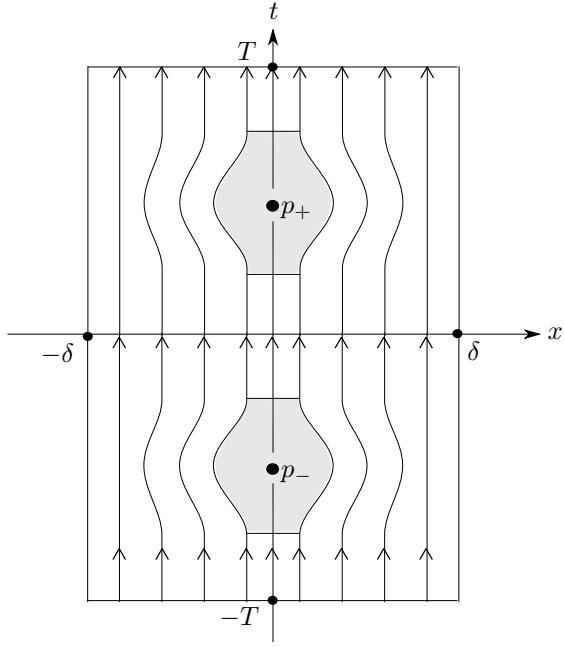
\caption{The rectangle $\Pi$, the levels of $H$ and the critical
  points $p_\pm=(0,\tau_\pm)$, and the neighborhoods $W$ (shaded) from
  the proof of Proposition \ref{prop:plugs}.}
\label{fig:plug}
\end{center}
\end{figure}

\begin{Lemma}
\label{lemma:plug}
Assume that the functions $H$ and $f$ satisfy
\ref{(P1)}--\ref{(P4)}. Then $X=\p/\p t$ near $\p P$. The (local) flow
of $X$ has exactly two integral curves entirely contained in $P$:
these are the periodic orbits $S^1\times \{p_\pm\}$. For every
integral curve which both enters and exits $P$ (i.e., meets the parts
of $\p P$ where $t=\pm T$), the exit and entrance points have the same
$x$ and $\theta$ coordinates. There exist ``trapped'' integral curves,
i.e., the integral curves that enter but do not exit the plug.
\end{Lemma}

\begin{proof}
  The first two assertions and also the last assertion readily follow
  from the explicit expression for $X$ given by \eqref{eq:X}. The
  third assertion is a straightforward consequence of the fact that,
  by \ref{(P4)}, the reflection in the $\{t=0\}$-plane changes the
  sign of the $(\theta,x)$-component of $X$; see, e.g., \cite[Sect.\
  2.2]{Gi:survey} for more details.
\end{proof}

When the functions $H$ and $f$ meet conditions \ref{(P1)}--\ref{(P4)},
we will refer to $P$ equipped with the form $\omega$ defined by
\eqref{eq:form} together with the embedding $j$ given by \eqref{eq:j}
as a \emph{symplectic plug} or, when the role of the parameters
$(\delta,T,a)$ and/or of the functions $H$ and $f$ is essential, as
the \emph{$(\delta,T,a; H,f)$-plug} or just the
\emph{$(H,f)$-plug}. The key to the proof of Theorem \ref{thm:main} is

\begin{Proposition}
\label{prop:plugs}
For any $\delta>0$ there exists a sequence of $(\delta,T_k,a_k;
H_k,f_k)$-plugs such that $T_k\to 0$ and $a_k\to 0$, and the plugs are
Hamiltonian diffeomorphic with diffeomorphisms equal to $\id$ near $\p
B$.
\end{Proposition}

\begin{Remark}
\label {rmk:plugs}
The last assertion of the proposition might, perhaps, require a
clarification. Without loss of generality we may assume that the
sequences $T_k\to 0$ and $a_k\to 0$ are strictly monotone
decreasing. Set 
$$
\Pi_k=[\delta,\,\delta]\times [-T_k,\,T_k],\textrm{ }
P_k=S^1\times\Pi_k\textrm{ and }B_k=P_k\times [-a_k,\,a_k]. 
$$
These are sequences of nested sets. Therefore, we can also view a
sequence of $(\delta,T_k,a_k; H_k,f_k)$-plugs as defined on the same
set $P=P_1$ and $B=B_1$, but with 
$$
\supp (H_k-x)\subset \Pi_k\textrm{ and }\|f_k\|_{C^0}<a_k.
$$ 
Furthermore, denote by $j_k$ the embedding $j$ of $P_k$ into $B_k$
given by \eqref{eq:j}. Likewise, we can interpret $j_k$ as a map from
$P$ to $B$ by extending it from $P_k$ to $P$ as $(\id,0)$. Then the
proposition asserts that, for a suitably chosen sequence of plugs,
there exists a sequence of Hamiltonian diffeomorphisms $\eta_k\colon
B\to B$ equal to $\id$ near $\p B$ and sending the image of $j$ to the
image of $j_k$. (Note that $\eta_k$ is not required, in any sense, to
conjugate the maps $j$ and $j_k$.)
\end{Remark}

\begin{Remark}
\label{rmk:2vs1}
For our purposes it would be sufficient to have such a sequence of
plugs with only one parameter, say $a$, going zero, while the other
one, $T$, remaining fixed. However, we find the (superficially)
stronger version of the proposition stated above more intuitive. Also,
the stronger version comes essentially for free as an easy consequence
of its one-parameter counterpart: to make $T$ and $a$ both small,
rather than just $a$, it suffices to apply a (suitably cut-off)
hyperbolic transformation in the $(t,y)$-plane; see the proof of the
proposition in Section \ref{sec:pf_plugs}.
\end{Remark}

\subsection{Proofs of Theorems \ref{thm:main} and \ref{thm:converge}}
Assuming Proposition \ref{prop:plugs} and postponing its proof to the
next section, let us now prove Theorems \ref{thm:main} and
\ref{thm:converge}.

\begin{proof}[Proof of Theorem \ref{thm:main}]
  The argument follows the line of reasoning outlined in Section
  \ref{sec:outline}.  Set $2n=4$ and let $S^{3}$ be the unit sphere in
  $\R^{4}$ with Darboux coordinates $(p_1,q_1,p_2,q_2)$. Consider the
  parallel transport along the $q_1$-axis in small $\eps>0$, i.e., the
  map
$$
(p_1,q_1,p_2,q_2)\mapsto (p_1,q_1,p_2,q_2)+w, \textrm{ where }
w=(0,\eps,0,0).
$$
After cutting off outside a neighborhood of $S^{3}$, we can view this
map as the Hamiltonian diffeomorphism $\varphi_{\tF}$ generated by the
Hamiltonian $\tF=\eps\chi\cdot p_1$, where $\chi$ is a cut-off
function equal to one near $S^{3}$, on a shell which contains both
$S^3$ and $S^3+w$.

For $\eps>0$ small, the only leafwise intersections of $S^{3}$ and
$\varphi_{\tF}(S^{3})$ are the two points $z^\pm$ on the unit circle
$S$ in the $(p_1,q_1)$-plane, located near the North Pole,
$(1,0,0,0)$, and the South Pole, $(-1,0,0,0)$, of $S^{3}$ and such
that
$$
\{\varphi_{\tF}(z^+),\varphi_{\tF}(z^-)\}=S\cap (S+w).
$$
More precisely, $z^\pm=(\pm \sqrt{1-\eps^2/4}, -\eps/2,0,0)$.

To see that there are no other leafwise intersections, note first that
$z$ is a leafwise intersection if and only if $z$ and
$\varphi_{\tF}(z)=z+w$ lie on the same Hopf circle. This Hopf circle
is then the unit circle in the real 2-plane spanned by $z$ and $z+w$,
obviously containing $w$. But this plane must also be a complex line
in $\R^{4}=\C^2$, and hence it must also contain $iw$. This forces the
plane to be the $(p_1,q_1)$-coordinate plane and $z$ to be one of the
points $z^\pm$.

Fix now an embedding $\psi$ (or $\psi^+$) of $S^1$ into a small
neighborhood of the North Pole in the intersection of $S^3$ with the
hyperplane $q_1=0$, sending $\theta=0$ to the North Pole.  For a
sufficiently small $\delta>0$, we can extend this embedding to a
symplectic embedding $\Psi$ (or $\Psi^+$) of $S^1\times
[-\delta,\delta]$, equipped with the symplectic form $\sigma=dx\wedge
d\theta$, into $S^3\cap\{q_1=0\}$.  Such an extension exists because
$S^3\cap\{q_1=0\}$ is symplectic near the North Pole. (Note also that
$\Psi$ does not pass through $z^+$, since $z^+$ lies in the
$q_1=-\eps/2$ plane.)

Furthermore, let $U=U^+$ be a neighborhood of $Y=\Psi\big( S^1\times
[-\delta,\delta]\big)$ in $\R^4$. We assume that $U$ is so small that
the following conditions are met:
\begin{itemize}
\item $z^+\not\in U$;
\item $\varphi_{\tF}^\tau(U)$ is contained in the region where
  $\chi\equiv 1$ for all $\tau\in [0,\,1]$;
\item $U$ is displaced by $\varphi_{\tF}$, i.e., $U\cap
  (U+w)=\emptyset$;
\item $U$ does not intersect the shifted sphere
  $\varphi_{\tF}(S^3)=S^3+w$.
\end{itemize}
The last two conditions require $U$ to be ``narrow and low''.

When the parameters $T$ and $a$ of the plug are small, we can
symplectically embed $B=P\times [-a,\,a]$ into $U$ extending the
embedding $\Psi$ of $S^1\times [-\delta,\,\delta]\times
\{0\}\times\{0\}$ to $B$ and sending $P$ into $V=U\cap S^3$. Since the
embedding is symplectic, it sends the characteristics in $P$ to the
characteristics in $V\subset S^3$. In particular, the $t$-axis in $P$
(i.e., the line $x=0=\theta$) matches the characteristic through the
North Pole (or, equivalently, through $z^+$) in $S^3$. From now on, we
identify $P$ and $B$ with their images under this embedding.

Such an embedding does exist because $\Psi$ is symplectic. To be more
precise, due to this assumption and a variant of the symplectic
neighborhood theorem, we may without loss of generality assume that
$U$ has the form
$$
U=
S^1\times (-\delta',\,\delta')\times (-T',\,T') \times (-a',\,a') 
$$ 
with coordinates $(\theta,x',t',y')$ and the symplectic form
$$
\sigma'=dx'\wedge d\theta  + dy'\wedge dt'
$$ 
and that $V\subset U$ is given by the condition $y'=0$ and the North
Pole is the origin $(0,0,0,0)$ in these coordinates. Furthermore, we
can also assume that $\Psi(\theta,x)=(\theta,x,0,0)$.  (Hence,
$\delta<\delta'$.) Now the required embedding, pulling back $\sigma'$
to $\sigma$, is
$$
(\theta,x,t,y)\mapsto (\theta,x',t',y') = (\theta,x,\kappa t,y/\kappa)
$$
where $\kappa$ is fixed and the positive parameters $T$ and $a$ are
chosen so small that and $\kappa T< T'$ and $a/\kappa< a'$. (In fact,
since we are free to chose any $\kappa$, it would be sufficient to
vary only one of the parameters $a$ and $T$. For instance, having $T$
fixed, we can take $\kappa=T'/2T$ and then $a$ so small that
$a/\kappa=2aT/T'<a'$.)

It is essential for what follows that in this construction $\delta'$,
and hence $\delta$, can be taken independent of $\eps$, while $T'$ and
$a'$ (and thus $T$ and $a$ or at least one of these parameters) are
bounded from above by some functions of $\eps$.

Next, we repeat this process starting with embeddings $\psi^-$ of
$S^1$ and $\Psi^-$ of $S^1\times [-\delta,\,\delta]$ into a
neighborhood of the South Pole in $S^3\cap\{q_1=0\}$ and passing
through the South Pole. As a result, we have an embedding of $B$ into
a small neighborhood $U^-$ of the band $Y^-=\Psi^-\big(S^1\times
[-\delta,\,\delta]\big)$ in $\R^4$, which sends $P$ into $V^-=U^-\cap
S^3$.

Replacing $P^\pm=P$ by $Q^\pm=j(P)$ in both neighborhoods $U^\pm$, we
obtain a new hypersurface $M$ which differs from $S^3$ only within
$U^\pm$ and is $C^0$-close to $S^3$; see Fig.\ \ref{fig:shift}. We
claim that there are no leafwise intersections of $M$ and
$\varphi_{\tF}(M)$.

To prove this, we need to show that for any characteristic $\Cc$ on
$M$, the sets $\Cc$ and $\varphi_{\tF}(\Cc)=\Cc+w$ do not
intersect. Observe that $U^\pm$ have been chosen to be so small that
the sets $\varphi_{\tF}(U^\pm)=U^\pm+w$ do not intersect $M$.  In
particular, by our choice of $U^\pm$, we have $M\cap (M+w)= S^3\cap
(S^3+w)$, which is the intersection of $S^3$ with the hyperplane
$q_1=\eps/2$. Therefore, all leafwise intersections on $M$ must be
outside $U^\pm$. Every characteristic $\Cc$ on $M$ comprises
(possibly) some arcs of a Hopf circle and (possibly) characteristics
in $Q^\pm$. We emphasize that, by Lemma \ref{lemma:plug}, the arcs lie
on the same Hopf circle. As a consequence, a leafwise intersection
$z\in\Cc$ can only be located on an arc of a Hopf circle and must also
be a leafwise intersection for $S^3$. Hence $z=z^\pm$. These two
points lie on the same characteristic $\Cc_0$ on $M$. This
characteristic is the union of an arc of the unit circle in the
$(p_1,q_1)$-plane contained in the $q_1>0$ half-plane and two
characteristics in $Q^\pm$, which are ``trapped'', i.e., enter $Q^\pm$
but never leave. Clearly, $\Cc_0$ does not intersect
$\varphi_{\tF}(\Cc_0)=\Cc_0+w$, and hence $z^\pm$ are not leafwise
intersections of $M$ with $\varphi_{\tF}(M)$.

Let us now consider a sequence $\eps_k\to 0$, set
$\tF_k=\eps_k\chi\cdot p_1$, and carry out the above construction for
every $k$. Namely, with embeddings $\Psi^\pm$ fixed (independent of
$k$), we can, as above, take a sequence of smaller and smaller
neighborhoods $U^\pm_k$ of the bands $Y^\pm$ and a nested sequence of
symplectic embeddings of $B_k=P_k\times [-a_k,a_k]$ into $U^\pm_k$
with $T_k\to 0$ and $a_k\to 0$. We again identify $B_k$ and $P_k$ with
their images $B^\pm_k$ and, respectively, $P^\pm_k$ in
$U^\pm_k$. Denote by $Q^\pm_k$ the images of $j_k(P_k)$ in $B^\pm_k$
and by $M_k$ the hypersurface obtained from $S^3$ by replacing
$P_k^\pm$ by $Q^\pm_k$. By construction, $M_k$ and
$\varphi_{\tF_k}(M_k)$ have no leafwise intersections.

By Proposition \ref{prop:plugs} and Remark \ref{rmk:plugs}, this can
be done so that there exist Hamiltonian diffeomorphisms $\eta_k^\pm$
of $B_1^\pm$, equal to $\id$ near the boundary and sending $Q^\pm_1$
to $Q^\pm_k\cup \big(P_1^\pm\setminus P_k^\pm\big)$. We extend
$\eta_k^\pm$ to a Hamiltonian diffeomorphism $\eta_k$ of $\R^4$ as the
identity map outside $B^\pm_1$. Setting $M=M_1$, we have
$\eta_k(M)=M_k$. Let $F_k=\tF_k\circ \eta_k$. Then $M$ and
$\varphi_{F_k}(M)$ have no leafwise intersections. Furthermore, $\|
F_k\|_{C^0}=\eps_k\to 0$ and $\supp F_k=\supp\chi$. Thus we can ensure
that $\supp F_k$ is contained in an arbitrarily small (but fixed,
i.e., independent of $k$) neighborhood of $S^3$ and that $M$ is
$C^0$-close to $S^3$.
\end{proof}

\begin{Remark}[Stability]
\label{rmk:stability}
The hypersurface $M$ constructed in the proof of Theorem
\ref{thm:main} is not stable in the sense of \cite{HZ}. Indeed, one of
the equivalent definitions of stability is that there exists a
one-form $\alpha$ on $M$ non-vanishing on the characteristic foliation
and such that $\ker\omega\subset \ker d\alpha$, where $\omega$ is the
restriction of the ambient symplectic form to $M$; see, e.g.,
\cite{Gi:coiso,EKP} for a discussion of stability. Then it suffices to
show that the plug $(P,\omega)$ with $\omega$ given by \eqref{eq:form}
is not stable in the sense of this definition. Arguing by
contradiction, assume that such a form $\alpha$ exists. The curves
$\gamma_\pm=S^1\times \{p_\pm\}$ form the boundary of the cylinder
$$
\Sigma=S^1\times \{0\}\times [\tau_-,\,\tau_+]\subset P
$$ 
foliated by the characteristics of $\omega$. Hence, due to the
condition $\ker\omega\subset \ker d\alpha$, we have
$d\alpha|_\Sigma=0$. Therefore, by Stokes' theorem,
$$
\int_{\gamma_-}\alpha=\int_{\gamma_+}\alpha,
$$
where $\gamma_\pm$ are oriented by fixing an orientation of $S^1$. On
the other hand, the direction of the vector field $X$ defined by
\eqref{eq:X} matches this orientation on one of the curves
$\gamma_\pm$ and the opposite orientation on the other.  Thus, since
$\alpha(X)\neq 0$ everywhere, these two integrals are non-zero and
have opposite signs. Note that, as a consequence, none of the
hypersurfaces $M_k$ is stable.  Furthermore, the characteristic
foliation of $M_k$ (or $M$) is not totally geodesic since for
hypersurfaces this requirement is equivalent to stability; see
\cite{Us}.
\end{Remark}

\begin{proof}[Proof of Theorem \ref{thm:converge}] It is clear that
  the sequence of hypersurfaces $M_k$ constructed in the proof of
  Theorem \ref{thm:main} $C^0$-converges to the round sphere
  $S^{3}$. The characteristic foliation of $M_k$ has non-compact
  leaves, e.g., the characteristics which enter and remain trapped in
  the plugs. Thus this foliation is not homeomorphic to the
  characteristic foliation on $S^{3}$ since the latter is just the
  Hopf fibration. Finally, working, say, in a neighborhood of the
  North Pole, we can take one of the two characteristics $j_k(S^1
  \times p_{\pm})$ as $L_k$, where $p_\pm$ are the critical points of
  $H$. These characteristics $C^\infty$-converge to the embedded
  circle $\Psi(S^1\times \{0\})\subset S^{3}$ which is clearly nowhere
  tangent to the Hopf fibration and intersects every fiber at at most
  one point.
\end{proof}

\begin{Remark}
  When $2n\geq 6$, we could have also used, with only minor
  modifications to the proof, more elaborate plugs from the
  constructions of counterexamples to the Hamiltonian Seifert
  conjecture; see, e.g., \cite{Gi:survey} and references therein. This
  would result in hypersurfaces $M$ without closed characteristics and
  leafwise intersections.
\end{Remark}

\section{Proof of Proposition \ref{prop:plugs}}
\label{sec:pf_plugs}

Throughout the proof, we fix positive parameters $\delta$, $T$ and
$a$. It is sufficient to find a sequence of functions $H_k$ and $f_k$
satisfying \ref{(P1)}--\ref{(P4)} and such that
\begin{enumerate}
\item[\reflb{itm:(i)}{(i)}] $\|f_k\|_{C^0}\to 0$ and
\item[\reflb{itm:(ii)}{(ii)}] $\supp (H_k-x)\subset
  [\delta,\,\delta]\times [-T_k,\,T_k]$ for some sequence $T_k\to 0$,
\end{enumerate}
and that the resulting plugs are Hamiltonian diffeomorphic. The last
requirement means that, in the notation from Section \ref{sec:plugs},
there exists a sequence of Hamiltonian diffeomorphisms of $B$ equal to
$\id$ near $\p B$ and sending $Q=j_1(P)$ to $Q_k=j_k(P)$; see Remark
\ref {rmk:plugs}.

We start by fixing a function $H$ and considering a family of
functions $f=f_s(x,t)$, depending smoothly on $s\in [0,\,1]$, and
meeting conditions \ref{(P1)}--\ref{(P4)} for every $s$. Moreover, we
require \ref{(P1)} to hold uniformly in $s$, i.e., that $f_s\equiv 0$
for all $s$ on some neighborhood of $\p \Pi$ independent of
$s$. Furthermore, we assume that
\begin{itemize}
\item[\reflb{itm:(F)}{(F)}] $f_s(x,t)=f_0(x,t)+c(s)$ near the points
  $p_\pm=(0,\tau_\pm)$, where $c(s)$ is a function of $s$ only.
\end{itemize}
This condition plays a crucial role in the proof. Note that here
again, \ref{itm:(F)} is required to hold uniformly in $s$, i.e., each
of the points $p_\pm$ has a neighborhood independent of $s$ where
$f_s=f_0+c(s)$.  (In what follows, we will always require conditions
of this type to hold uniformly, without mentioning this specifically
again.)

We will show that the resulting family of plugs is Hamiltonian
diffeomorphic. This is clearly enough to construct a sequence of
Hamiltonian diffeomorphic plugs satisfying \ref{itm:(i)} and having
$H$ fixed.  Then, as was pointed out in Remark \ref{rmk:2vs1}, it is
easy to have \ref{itm:(ii)} also satisfied by applying a cut-off
hyperbolic transformation in the $(t,y)$-plane. We will carry out the
argument in detail at the end of the proof.

Denote by $j_s$ the embedding $P\to B$ given by \eqref{eq:j} for
$(H,f_s)$, and set
$$
\omega_s=j_s^*\sigma=d(H\,d\theta-f_s\,dt).
$$
The key step of the proof is the following result showing that these
forms are diffeomorphic to each other:

\begin{Lemma}
\label{lemma:psi}
There exists an isotopy $\psi_s$, $s\in [0,\,1]$, of $P$ equal to
$\id$ on a neighborhood of $\p P$ and such that
$\psi_s^*\omega_s=\omega_0$.
\end{Lemma} 

\begin{proof}
  We use Moser's homotopy method. It suffices to find an $s$-dependent
  vector field $Z_s$ generating $\psi_s$. As is easy to see, $Z_s$
  then must satisfy the equation
\begin{equation}
\label{eq:Z}
L_{Z_s}\omega_s= d\left(\frac{\p f_s}{\p s}\right)\wedge dt .
\end{equation}
We look for a solution $Z_s$ of the form
$$
Z_s=g_s\frac{\p}{\p \theta},
$$
where $g_s$ is an (unknown) function on $\Pi=[-\delta,\,\delta]\times
[-T,\,T]$ depending on the parameter $s\in [0,\,1]$ and vanishing near
$\p \Pi$.

Calculating the right and the left hand sides of \eqref{eq:Z}
explicitly, we arrive at the equation
$$
\big((g_s)'_t H'_x-(g_s)'_x H'_t\big)\, dx \wedge dt 
= F_s\, dx \wedge dt,
$$
where
$$
F_s=\frac{\p }{\p s}\left(\frac{\p f_s}{\p x}\right).
$$
Equivalently, this equation can be rewritten as
\begin{equation}
\label{eq:g}
L_{\xi_H} g_s= F_s,
\end{equation}
where $\xi_H$ is the Hamiltonian vector field of $H$ on $\Pi$.  We
need to find a solution $g_s$ of \eqref{eq:g} which vanishes on a
neighborhood of $\p \Pi$ for all $s$. The right hand side $F_s$ of
\eqref{eq:g} has the following properties:
\begin{itemize}
\item $F_s$ is odd in $t$ for every $s$ and $x$ (by \ref{(P4)});
\item $F_s$ vanishes for all $s$ on some neighborhoods of the zeros
  $p_\pm$ of $\xi_H$ (by \ref{itm:(F)});
\item $F_s$ vanishes for all $s$ on a neighborhood of the boundary $\p
  \Pi$.
\end{itemize}
These are the only features of $F_s$ needed for the proof. (In
contrast, a rather specific form of $\xi_H$ determined by
\ref{(P1)}--\ref{(P4)} is essential.)

We will solve equation \eqref{eq:g} using the method of
characteristics; see, e.g, \cite[Chap.\ 2]{Ar}. A somewhat non-obvious
point is then the existence and smoothness of a solution near the
zeros $p_\pm$ of $\xi_H$. Thus, before turning to the task of solving
the equation, let us state and prove the existence and smoothness
criterion we will use. (We are assuming below that $H$ is
$C^\infty$-smooth.)

Let $W$ be a small neighborhood of one of the points $p_\pm$. For the
sake of simplicity, let us also assume that $W$ is cut out by two
horizontal lines $t=\const$ at the top and the bottom and by two
levels of $H$ on its sides; see Fig.\ \ref{fig:plug} and Fig.\
\ref{fig:W}. In particular, the intersection of $W$ with every
integral curve of $\xi_H$ is connected. Furthermore, assume that $g$
is a function on the slit neighborhood
$\accentset{\circ}{W}=W\setminus \big(\{0\}\times\{t\geq
\tau_\pm\}\big)$ constant along every integral curve of $\xi_H$ in
this set. (We think of $g$ as a possibly non-smooth solution of the
homogeneous equation $L_{\xi_H} g = 0$ in $W$.)  Finally, let us
require the restriction of $g$ to the cross section
$\Gamma=W\cap\{t=t_0\}$ to be $C^k$-smooth, where $t_0<\tau_\pm$ is
close to $\tau_\pm$; see Fig.\ \ref{fig:W}. (Thus $\Gamma$ intersects
all levels of $H$ in $W$.) Then we claim that $g$ extends to a
$C^k$-smooth function on $W$. Such an extension is automatically
unique and constant along the levels of $H$ in $W$. In particular, $g$
can be thought of as a true solution of the homogeneous equation in
$W$.

\begin{figure}[h!]
\begin{center} 
\def\svgwidth{0.4\columnwidth}
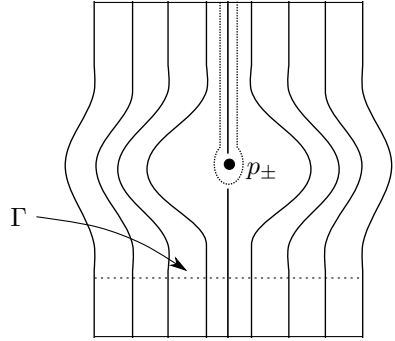
\caption{Neighborhood $W$, the slit $\{0\}\times\{t\geq\tau_{\pm}\}$
  and the cross section $\Gamma$.}
\label{fig:W}
\end{center}
\end{figure}

Let us prove this claim. By definition, the value of the extension at
a point $z\in W$ is equal to the value of $g$ at the point $z'$ where
the level of $H$ through $z$ intersects $\Gamma$. (A continuous
extension, if it exists, must have this property. Note also that,
since we only need to extend $g$ to the slit $\{0\}\times\{t\geq
\tau_\pm\}$, we could have just set $g\equiv g(0,t_0)$ on the cut.)
To see that the extension defined in this way is $C^k$-smooth, we
observe first that, by \ref{(P2)} and \ref{(P3)}, $H'_x>0$ along
$\Gamma$. Thus the restriction of $H$ to $\Gamma$ is a diffeomorphism
onto its image $H(\Gamma)=H(W)$. Let us denote by
$(H|_\Gamma)^{\mathrm{inv}}$ the inverse diffeomorphism. This map is
as smooth as $H$, and hence $C^\infty$-smooth. Now it is easy to see
that the extension is given by $g|_\Gamma\circ
(H|_\Gamma)^{\mathrm{inv}}\circ H$ on $W$. Indeed, this composition is
equal to $g$ on $\accentset{\circ}{W}$. Furthermore, every function in
the composition is at least $C^k$-smooth. Hence the extension is also
$C^k$-smooth on~$W$.

Of course, there is also an obvious parametric version of this
result. Namely, now the function $g$ depends on a parameter $s\in
[0,\,1]$ and $g|_\Gamma$ is $C^k$ in $x$ and $s$. Then $g$ is $C^k$ on
$W\times [0,\,1]$.

With these observations in mind we are ready to solve \eqref{eq:g}. We
do this in several steps. For the sake of simplicity, we suppress the
parameter $s$ in the notation.

We set the ``initial condition'' $g\equiv 0$ on the bottom part
$\{t=-T\}$ of $\p\Pi$ in the Cauchy problem \eqref{eq:g}. The method
of characteristics guarantees the existence of a smooth solution $g$
on the cut domain $\Pi\setminus \big(\{0\}\times [\tau_-,\,T]\big)$;
see, e.g., \cite[Chap.\ 2]{Ar}.  By the above criterion, since
$F_s\equiv 0$ near $p_-$, the solution $g$ extends to the union of the
cut domain with a neighborhood of $p_-$. Applying the method of
characteristics again with the initial condition on the horizontal
line just above $p_-$, we can extend $g$ to the domain $\Pi\setminus
\big(\{0\}\times [\tau_+,\,T]\big)$. Next, again, by the smoothness
criterion, the solution extends to a neighborhood of $p_+$ and then by
the method of characteristics to the entire domain $\Pi$.

Since $\xi_H=\p/\p t$ and $F\equiv 0$ near $\p\Pi$, it is clear that
the resulting solution $g$ vanishes (for all $s$) near $\p \Pi$ except
possibly the upper part $\{t=T\}$ of the boundary. Recall, however,
that $H$ is even in $t$ by \ref{(P4)} and $F$ is odd in $t$. It
readily follows then that the solution $g$ is even in $t$, and hence
it also vanishes near the upper part of the boundary. \end{proof}

Our next step is to show that there exists an $s$-dependent
Hamiltonian $K_s$ on $B$ with $j_s(P)=\varphi_{K}^s(j(P))$. In fact,
we will find $K_s$ such that
\begin{equation}
\label{eq:K}
\varphi_K^s|_{j(P)} =j_s\psi_s
\colon P\to B
\end{equation}
In addition, $K_s$ is required to vanish uniformly in $s$ near $\p B$.

Consider the vector field
$$
v_s=\frac{\p}{\p s} j_s\psi_s
$$
along the embedding $j_s\psi_s$. Then \eqref{eq:K} is satisfied if and
only if
$$
dK_s=
-i_{v_s}\sigma
$$
along $j_s\psi_s$, i.e., at every point of $TB|_{j_s\psi_s(P)}$. It is
not very difficult to show that $K_s$ with this property exists if and
only if the form
$$
\beta_s=(j_s\psi_s)^*i_{v_s}\sigma
$$
is exact on $P$.

Let us calculate this form explicitly. Consider a primitive $\lambda$
of $\sigma$ on $B$, e.g., we can take $\lambda=x\,d\theta+y\,dt$. Then
extending $v_s$ to an $s$-dependent vector field on the entire domain
$B$, we have
\begin{equation}
\label{eq:beta}
\beta_s=(j_s\psi_s)^*L_{v_s}\lambda-(j_s\psi_s)^* d(\lambda(v_s)).
\end{equation}
Hence it suffices to show that the form
$$
\alpha_s=(j_s\psi_s)^* L_{v_s}\lambda 
= \frac{d}{d s}(j_s\psi_s)^*\lambda
$$
is exact on $P$. By Lemma \ref{lemma:psi}, $\psi_s^*j_s^*\, d\lambda$
is independent of $s$, and hence $\alpha_s$ is closed for all $s$. To
see that this form is exact it suffices to show that its integral over
any circle $S^1\times \{(x,t)\}$ is zero. But $\alpha_s\equiv 0$ near
$\p P$ because $j_s=(0,\id)$ and $\psi_s=\id$ near $\p P$. Thus the
integral is zero when $(x,t)$ is near $\p P$ (and therefore for any
circle) and the form is exact.

Finally, to have $K_s$ vanishing near $\p B$ (for all $s$) it suffices
to show that $\beta_s$ has a primitive vanishing near $\p P$. We have
already seen that $\alpha_s$, the first term in \eqref{eq:beta}, has
such a primitive since it vanishes near $\p P$. A primitive of the
second term in \eqref{eq:beta} is $(j_s\psi_s)^* \lambda(v_s)$. This
function is also identically zero near $\p P$ because $j_s=(0,\id)$
and $\psi_s=\id$, and hence $v_s\equiv 0$, near the boundary.

Now, let $(H,f)$ be a pair of functions meeting conditions
\ref{(P1)}--\ref{(P4)}. As is easy to see, we can always find a
function $\tf$ with arbitrarily small $C^0$-norm on $\Pi$ and a family
$f_s$ connecting $f_0=f$ and $f_1=\tf$, satisfying condition
\ref{itm:(F)} and such that the pair $(H,f_s)$ also meets the plug
requirements \ref{(P1)}--\ref{(P4)} for all $s$ (uniformly in $s$). It
follows that $(H,\tf)$-plug is Hamiltonian diffeomorphic to the
original $(H,f)$-plug.

In other words, given an $(H,f)$-plug, we can find an $(H,\tf)$-plug,
Hamiltonian diffeomorphic to the original plug, with arbitrarily
$C^0$-small $\tf$. To finish the proof of Proposition
\ref{prop:plugs}, we need to find an $(\hH,\hf)$-plug with both
$\|\hf\|_{C^0}$ and the $t$-component of $\supp (\hH-x)$ arbitrarily
small, which is Hamiltonian diffeomorphic to the $(H,f)$-plug.

We have $\|\tf\|_{C^0}\leq \eps$, where $\eps>0$ can be assumed to be
arbitrarily small, and
$$
\supp (H-x)\subset P'=S^1\times \Pi', \textrm{ where } \Pi'=[-\delta',\,\delta']\times
[-T',\,T']
$$
for some positive $\delta'<\delta$ and $T'<T$. Furthermore, set
$B'=P'\times [-a',\,a']$ with $\eps<a'<a$, and consider the
Hamiltonian
$$
G=-\kappa\cdot yt\cdot b(x,t,y)
$$
with $\kappa>0$ and the cut-off function $b$ is equal to one on $B'$
and zero near $\p B$. The Hamiltonian vector field $\xi_G$ is simply
the hyperbolic vector field $(-\kappa t, \kappa y)$ in the
$(t,y)$-plane as long as $(\theta,x,t,y)\in B'$. (Our sign convention
for the Hamilton equation is $i_{\xi_G}\sigma=-dG$.) Furthermore, as
is easy to see, the flow of $G$ preserves the subset $P=\{y=0\}\subset
B$.

Assume now that $e^\kappa \eps<a'$ and denote by $\tj$ the embedding
$j$, given by \eqref{eq:j}, for the $(H,\tf)$-plug. Then $\varphi_G$
sends the image of $\tj$ to the image of $\hj$, the $(\hH,\hf)$-plug
embedding, with
$$
\hH(x,t)= H(x,e^\kappa t)
$$
and 
$$
\hf(x,t)=e^{\kappa} f(x,e^\kappa t).
$$
In these formulas we treat $(x,t)$ as a point in $\R^2$ rather than in
$\Pi$. However, $\hH$ and $\hf$ still satisfy \ref{(P1)}--\ref{(P4)}
with the same $\delta$, $T$ and $a$ as $(H,f)$ and $(H,\tf)$, and
hence give rise to a plug with the same sets $P$ and $B$ as the
$(H,\tf)$-plug and the $(H,f)$-plug. Furthermore,
$$
\|\hf\|_{C^0}<e^{\kappa}\eps \textrm{ and }
\supp(\hH-x)\subset S^1\times [-\delta,\,\delta]\times
[-e^{-\kappa}T,\, e^{-\kappa}T].
$$
The only constraint on $\kappa>0$ and $\eps>0$ is that
$e^{\kappa}\eps<a'$. Thus, by choosing a sufficiently large $\kappa$
and then a sufficiently small positive $\eps$, we can make
$e^{\kappa}\eps$ and $e^{-\kappa}T$ arbitrarily small. By
construction, the $(H,f)$-plug, the $(H,\tf)$-plug, and the
$(\hH,\hf)$-plug are Hamiltonian diffeomorphic. It follows that
starting with an arbitrary $(H,f)$-plug one can find a sequence of
$(H_k,f_k)$-plugs (with $H_1=H$ and $f_1=f$), Hamiltonian
diffeomorphic to each other and satisfying \ref{itm:(i)} and
\ref{itm:(ii)}. This completes the proof of Proposition
\ref{prop:plugs}.  \hfill $\qed$

\end{document}